\documentclass[a4paper,11pt,oneside]{amsart}
\pdfoutput=1

\usepackage[utf8]{inputenc}
\usepackage{bm}
\usepackage{mathtools,amssymb}
\usepackage{esint}
\usepackage{hyperref}
\hypersetup{
    colorlinks=true,
    linkcolor=black,
    filecolor=black,      
    urlcolor=cyan,
    citecolor=black
    }
\usepackage{tikz}
\usepackage{dsfont}
\usepackage{relsize}
\usepackage{url}
\usepackage{xcolor}
\usepackage{graphicx}
\usepackage{mathrsfs}
\usepackage[shortlabels]{enumitem}
\usepackage{lineno}
\usepackage{amsmath}
\usepackage{enumitem}
\usepackage{amsthm} 
\usepackage{verbatim}
\usepackage{dsfont}

\allowdisplaybreaks

\mathtoolsset{showonlyrefs}

\graphicspath{{images/}}

\newtheorem{theorem}{Theorem}[section]
\newtheorem{lemma}[theorem]{Lemma}

\newtheorem{proposition}[theorem]{Proposition}

\newtheorem{remark}[theorem]{Remark}

\title[Inverse fractional conductivity problem]{Partial data stability for the inverse fractional conductivity problem}
\author{Giovanni Covi}
\address{Department of Mathematics and Statistics, University of Helsinki, Finland}
\email{\texttt{giovanni.covi@helsinki.fi}} 
\author{Antti Kujanpää}
\address{Computational Engineering, School of Engineering Sciences,
Lappeenranta-Lahti University of Technology LUT, Lappeenranta, Finland}
\email{antti.kujanpaa@lut.fi}
\author{Jesse Railo}
\address{Computational Engineering, School of Engineering Sciences,
Lappeenranta-Lahti University of Technology LUT, Lappeenranta, Finland \& Department of Mathematics, Stanford University, Stanford, CA, USA}
\email{jesse.railo@lut.fi; railo@stanford.edu}
\date{\today}

\newcommand{\R}{{\mathbb R}}

\newcommand{\vev}[1]{\left\langle#1\right\rangle}

\newcommand{\norm}[1]{\lVert #1 \rVert}
\newcommand{\abs}[1]{\left\lvert #1 \right\rvert}
\newcommand{\ip}[2]{\left\langle #1,#2 \right\rangle}
\DeclareMathOperator{\Div}{div} 
\DeclareMathOperator{\supp}{supp} 

\begin{document}

\begin{abstract}
The classical Calderón problem with partial data is known to be log-log stable in some special cases, but even the uniqueness problem is open in general. We study the partial data stability of an analogous inverse fractional conductivity problem on bounded smooth domains. Using the fractional Liouville reduction, we obtain a log-log stability estimate when the conductivities a priori agree in the measurement set and their difference has compact support. In the case in which the conductivities are assumed to agree a priori in the whole exterior of the domain, we obtain a shaper logarithmic stability estimate.
\end{abstract}

\maketitle


\section{Introduction}
We study the stability properties of the fractional Calder\'on problem, and obtain logarithmic estimates from partial exterior data. This result improves the earlier state of the art (see \cite{doi:10.1137/22M1533542}), in the sense that the measurements are obtained only from a bounded subset of the exterior domain.

The interest in stability estimates for inverse problems arises from their real-life use in imaging techniques from the fields of medicine, engineering, and more. Due to the typical ill-posedness of many inverse problems, small measuring errors often produce large errors in reconstruction. The stability estimates give a way to theoretically quantify and control such errors. A prototypical example of such behaviour is given by the classical Calder\'on problem in electrical impedance tomography (see \cite{C80}, \cite{GU30years}), which asks to reconstruct the electrical conductivity within an object from measurements of voltage and currents on its boundary. The general, anisotropic problem is still wide open at the time of writing, but many steps forward have been made in numerous directions. We cite a few of them in order to give an overview. The fundamental uniqueness result has been obtained by Sylvester and Uhlmann in dimension $n \geq 3$, using the so called complex geometrical optics solutions \cite{SU87}. The case $n=2$ was studied by Astala and Päivärinta using complex analytical tools \cite{AP06}. A reconstruction method has been established by Nachman \cite{Na88}. In their fundamental works, Alessandrini \cite{Alessandrini-stability} and Mandache \cite{MandacheInstability} were able to show that logarithmic stability holds for the classical Calder\'on problem, and that this result is in fact optimal. The instability result has been improved and greatly generalized by Koch, Rüland and Salo (\cite{KochRulandSaloInstability}, see also the recent works \cite{CaroDSFerreiraRuiz, CaroSaloAnisotropic} for stability with partial data). In the case of piecewise constant conductivities, the stability result has been improved from logarithmic to Lipschitz \cite{AlessandriniVessellaLipschitz}, but the constants may grow rapidly when the number of pieces increases. The anisotropic Calderón problem has furhter been studied in \cite{AstalaPaivarintaLassasAnisotropic,FKS-Anisotropic}.

In this work, we study a fractional, nonlocal analogue of the classical Calder\'on problem, following a line of study initiated by Ghosh, Salo and Uhlmann in their seminal work \cite{GSU20}. In particular, we are interested in the stability properties of the fractional Calder\'on problem in its conductivity formulation. It is worth to notice that there exist other formulations as well, which are in general not equivalent to ours (see \cite{ghosh2021calderon} for a formulation via the heat semigroups, and  \cite{CGRU23} for one via the Caffarelli-Silvestre extension). Given $s\in(0,1)$, the exterior value problem for the fractional conductivity equation is
\begin{equation}\begin{cases}\label{nonlocal-conductivity}
    \mbox{div}_s (\Theta_\gamma\nabla^su)&=0 \quad \mbox{ in } \Omega, \\ u&=f \quad \mbox{ in } \Omega_e ,
\end{cases}\end{equation}
where $\Omega_e\vcentcolon=\R^n\setminus\overline{\Omega}$ is the exterior domain, and $\Theta_{\gamma}(x,y)\vcentcolon =\gamma^{1/2}(x)\gamma^{1/2}(y)Id$ is the conductivity matrix. The nonlocal operators $\mbox{div}_s, \nabla^s$ (defined in detail in Section \ref{sec: preliminaries}) are called fractional divergence and fractional gradient, and they are nonlocal counterparts of the differential operators $\mbox{div}, \nabla$, with which they share the properties $$(\mbox{div}_s)^*=\nabla^s, \qquad \mbox{div}_s\nabla^s = (-\Delta)^s.$$ The fractional conductivity operator is a map $$\Div_s(\Theta_\gamma \nabla^s) : H^s(\R^n)\rightarrow H^{-s}(\R^n).$$ We say that $u\in H^s(\R^n)$ is a weak solution of \eqref{nonlocal-conductivity} if $u-f \in \widetilde{H}^s(\Omega)$ and
\begin{equation}\label{eq:generalNonlocalOperators}
    B_\gamma(u,\phi):=\frac{C_{n,s}}{2}\int_{\R^{2n}} \frac{\gamma^{1/2}(x)\gamma^{1/2}(y)}{\abs{x-y}^{n+2s}} (u(x)-u(y))(\phi(x)-\phi(y))\,dxdy =0
\end{equation} 
holds for all $\phi \in C_c^\infty(\Omega)$. All the necessary Bessel potential spaces and details are defined in Section \ref{sec: preliminaries}. One can show that for all $f\in X\vcentcolon = H^s(\R^n)/\widetilde{H}^s(\Omega)$ there exists a unique weak solution $u_f\in H^s(\R^n)$ of the fractional conductivity equation \eqref{nonlocal-conductivity}. With this in mind, we define the exterior DN map $\Lambda_{\gamma}\colon X\to X^*$ by 
        \[
        \begin{split}
            \langle \Lambda_{\gamma}f,g\rangle \vcentcolon =B_{\gamma}(u_f,g).
        \end{split}
        \]
In the inverse problem for the fractional conductivity equation, one is asked to recover the conductivity $\gamma$ from (partial knowledge of) the map $\Lambda_\gamma$. This formulation of the fractional Calder\'on problem has been object of intensive study; we refer to \cite{covi-survey} for the most recent references. In particular, the uniqueness question for the fractional Schrödinger equation was considered already in \cite{GSU20}, and later extended in \cite{GLX-calderon-nonlocal-elliptic-operators, CLR18-frac-cald-drift,LL-fractional-semilinear-problems,CMR20,CMRU22,C21} to the cases of perturbations of nonlinear, higher order, or quasilocal kind, to name a few. A general theory for nonlocal elliptic equations was developed in \cite{RS-fractional-calderon-low-regularity-stability,RZ2022unboundedFracCald}. Global uniqueness for the fractional conductivity equation was obtained in \cite{RGZ2022GlobalUniqueness} (see also the earlier works \cite{covi2019inverse-frac-cond, RZ2022unboundedFracCald}) by means of the  fractional Liouville reduction (see Section \ref{sec: preliminaries}). The uniqueness question for low regularity was studied in \cite{RZ2022LowReg} by means of the UCP results of \cite{KRZ2022Biharm}. Many counterexamples to uniqueness were obtained for special geometries \cite{RZ2022counterexamples,RZ2022LowReg}. The problem of stability for the fractional conductivity equation has been the object of \cite{doi:10.1137/22M1533542}, where a logarithmic estimate was obtained under smallness assumptions for the difference of the DN maps in the case of complete data.  

\subsection{Main results}

Our first main result is a logarithmic stability estimate, which we obtain in the case of partial data and under the assumption that the unknown conductivities coincide a priori in the exterior set $\Omega_e$.

\begin{theorem}\label{thm: stability estimate partial I} 
Let $\Omega\subset \mathbb R^n$ be a bounded, smooth domain, let $0<s<\min(1,n/2)$, and assume that $W_1,W_2 \Subset \Omega_e$ are nonempty open sets. Let $\gamma_1,\gamma_2\in L^\infty(\R^n)$ be such that 
\begin{enumerate}[(i)]
        \item $\gamma_1=\gamma_2$ in $\Omega_e$,\label{item: agree on Ext} 
        \item\label{item: main assumption 1_first} $\gamma_0\leq \gamma_j(x) \leq \gamma_0^{-1}$, for $j=1,2$ and some $\gamma_0\in (0,1)$,
    \end{enumerate} 
and assume that there exist $\varepsilon, C_1>0$ such that the background deviations $m_1,m_2$ verify
\begin{enumerate}[(i)]
\setcounter{enumi}{2}
        \item\label{item: main assumption 2I0} $m_1,m_2\in  H^{2s+\varepsilon,\frac{n}{s}}(\R^n)$, with 
        \begin{equation}
        \label{eq: main bound 1I0}
            \|m_i\|_{H^{2s+\varepsilon,\frac{n}{s}}(\R^n)}\leq C_1, \qquad \mbox{for } i=1,2.
        \end{equation}
\end{enumerate}
Then 
    \begin{equation}
    \label{eq: stability estimate I}
       \| \gamma_1 - \gamma_2 \|_{H^s (\Omega)} \leq \omega(\|\Lambda_{\gamma_1} - \Lambda_{\gamma_2}\|_{\widetilde{H}^s(W_1) \to (\widetilde{H}^s(W_2))^*}),
    \end{equation}
     where $\omega$ is a logarithmic modulus of continuity satisfying
    \[
     \omega(t) \leq C|\log t|^{-\sigma},\quad\text{for}\quad 0 < t\leq \lambda,
    \]
    for three constants $\sigma,\lambda,C>0$ depending only on $s,\varepsilon,n,\Omega, W_1,W_2, C_1$ and $\gamma_0$.
\end{theorem}

The proof of this result is based on the fractional Liouville reduction, which allows us to restate the problem in terms of the fractional Schrödinger equation, and on the related result by R\"uland and Salo (see \cite{RS-fractional-calderon-low-regularity-stability}). However, the reduction is not immediately straightforward, and requires some regularity considerations. \\

For the case in which the two unknown conductivities are not known to concide a priori in the exterior, we have obtained the following log-log stability result. Here we assume that the difference of the conductivities is supported in a compact set $\Sigma$.

\begin{theorem}\label{thm: stability estimate partial II}
Let $\Omega\subset \mathbb R^n$ be a bounded, smooth domain, let $0<s<\min(1,n/2)$, and assume that $\Sigma \subset \R^n$ is a compact set and $W \subset \R^n$ is a nonempty bounded  domain such that $\overline W\cap ( \overline{\Omega} \cup \Sigma ) = \emptyset$. Assume that the conductivities $\gamma_1,\gamma_2\in L^\infty(\R^n)$ are such that 
\begin{enumerate}[(i)]
        \item  $\supp(\gamma_1-\gamma_2) = \Sigma$, \label{item: compact supp assumption}
        \item\label{item: main assumption 1_two} $\gamma_0\leq \gamma_j(x) \leq \gamma_0^{-1}$, for $j=1,2$ and some $\gamma_0\in (0,1)$,
    \end{enumerate} 
and assume that there exist $\varepsilon, C_1>0$ such that the background deviations $m_1,m_2$ verify
\begin{enumerate}[(i)]
\setcounter{enumi}{2}
        \item\label{item: main assumption 2I00} $m_1,m_2\in  H^{2s+\varepsilon,\frac{n}{s}}(\R^n)$, with 
        \begin{equation}
        \label{eq: main bound 1I00}
            \|m_i\|_{H^{2s+\varepsilon,\frac{n}{s}}(\R^n)}\leq C_1, \qquad \mbox{for } i=1,2.
        \end{equation}
\end{enumerate}    
   Then for all $1\leq p < \frac{2n}{n-2s}$ we have
    \begin{equation}
    \label{eq: stability estimate II}
       \| \gamma_1 - \gamma_2 \|_{L^{p}(\R^n)} \leq \omega(\|\Lambda_{\gamma_1} - \Lambda_{\gamma_2}\|_{\widetilde{H}^s(W) \to (\widetilde{H}^s(W))^*}),
    \end{equation} 
    where $\omega(t)$ is a log-logarithmic modulus of continuity satisfying
    \[
     \omega(t) \leq C_1'|\log(C_0'|\log t|^{-\sigma_0})|^{-\sigma_1},\quad\text{for}\quad 0 < t\leq \lambda,
    \]
    for constants $\sigma_0, \sigma_1, C'_0 ,C'_1,\lambda>0$ depending only on $s,p,W,\Sigma,\varepsilon,n,\Omega, C_1$ and $\gamma_0$.
\end{theorem}

    Observe that condition \ref{item: compact supp assumption} from Theorem \ref{thm: stability estimate partial II} generalizes condition  \ref{item: agree on Ext} from Theorem \ref{thm: stability estimate partial I}, which corresponds to the special case $\Sigma\subseteq\overline\Omega$. However, in both cases we have $\gamma_1=\gamma_2$ in $W$. In the proof of the result above, the log-log estimate is obtained by combination of the cited logarithmic estimate from R\"uland, Salo with the result \cite[Proposition 6.1]{GRSU-fractional-calderon-single-measurement}, which estimates a function in $\Omega$ in terms of its fractional Laplacian in the exterior.

\subsection{Organization of the article} The rest of the article is organized as follows. Section 2 is dedicated to the  preliminaries, including the definitions of several Bessel potential spaces, the fractional Schr\"odinger equation, and the main focus of our study: the fractional conductivity equation. The fractional Liouville reduction, which establishes a connection between the two, is also discussed in Section 2. Section 3 contains the proof of Theorem \ref{thm: stability estimate partial I}, which is obtained via several steps, each of which is stated as a separated lemma. Finally, in Section 4 we provide the proof of Theorem \ref{thm: stability estimate partial II}.

\subsection*{Acknowledgements} This work was supported by the Research Council of Finland through the Flagship of Advanced Mathematics for Sensing, Imaging and Modelling (decision numbers 359182 and 359183). G.C. was supported by an Alexander-von-Humboldt fellowship and by the Research Council Finland (CoE in Inverse Modelling and Imaging and FAME flagship, grants 353091 and 359208). J.R. was supported by Emil Aaltonen Foundation, Fulbright Finland Foundation (ASLA-Fulbright Research Grant for Junior Scholars 2024--2025), and Jenny and Antti Wihuri Foundation.

\section{Preliminaries}\label{sec: preliminaries}

\subsection{Bessel potential spaces}
Denote by $   \mathcal{F} $ and $\mathcal{F}^{-1}$ the Fourier and inverse Fourier transforms in $\R^n$, respectively. 
For $1\leq p < \infty$ and $s \in \R $, we define the Bessel potential space  $H^{s,p} ( \R^n ) $ as the set 
\[
H^{s,p} ( \R^n )  := \{ u \in L^p(\R^n) : \mathcal{F}^{-1} \left(  (1+ | \xi |^2 )^{\frac{s}{2}} \mathcal{F} u \right) \in  L^p(\R^n) \} ,
\]
equipped with the norm 
\[
\| u \|_{H^{s,p} ( \R^n ) }  :=  \left\|  \mathcal{F}^{-1} \left(  (1+ | \xi |^2 )^{\frac{s}{2}} \mathcal{F}   u\right)  \right\|_{L^p (\R^n)}.
\]
It is well known that $( H^{s,p} ( \R^n ) , \| \cdot \|_{H^{s,p} ( \R^n ) }  )$ is a Banach space. 
Denote $H^{s} ( \R^n )  := H^{s,2} ( \R^n ) $. 
For open $U \subset \R^n$, we define 
\[
H^{s,p} ( U )  := \{ u |_U : u \in H^{s,p} ( \R^n ) \}, \quad H^{s} ( U ) := H^{s,2} ( U ),
\]
and associate it with the quotient norm 
\[
\| u \|_{H^{s,p} ( U ) } := \inf \{ \| v \|_{ H^{s,p} ( \R^n )  } : v \in H^{s,p} ( \R^n ), \ v|_U = u \}.
\]
The closure of $C_c^\infty (U) $ in $H^{s,p}(\R^n)$ shall be denoted by $\widetilde{H}^{s,p} ( U )$. Again, $\widetilde{H}^{s} ( U ) : =  \widetilde{H}^{s,2} ( U )$.
The space $H^{s}(U)$ coincides (see \cite[Theorem 3.3]{MR3620760}) with the dual space 
$(\widetilde{H}^{-s}(U))^*$.

\subsection{The fractional Schrödinger equation}

Let $r\in\mathbb R^+$. The fractional Laplacian $(-\Delta)^r  : H^{s,p} ( \R^n ) \to H^{s-2r,p} ( \R^n ) $, $1\leq p $, $r \in \R$ is defined by 
\[
(-\Delta)^r u := \mathcal{F}^{-1} \left( |\xi|^{2r} \mathcal{F} u \right). 
\]
For $r = 1$ it coincides with the standard Laplacian $-\Delta = - \sum_{j=1}^n \partial_j^2 $, which is a local operator. For all $r\in\mathbb Z^+$, the operator $(-\Delta)^r$ is also local; however, for $r \in \R^+\setminus \mathbb Z$ the fractional Laplacian does not preserve supports, i.e. it is a nonlocal operator. \\

Let $\Omega \subset \R^n$ be a bounded  smooth domain, and denote by $\Omega_e$ the exterior domain $ \R^n \setminus \overline\Omega$.  
Let $0< s < 1$ and consider solutions $u \in H^s (\R^n)$ of the fractional Schrödinger equation
\[
    (-\Delta)^s u + q  u = 0\quad\text{in}\quad  \Omega,
\]
where the potential $q$ belongs to $L^{\frac{n}{2s}}( \R^n )$. 
The function $u \in H^s(\R^n)$ is said to be a weak solution of this equation if  $B_q(u,\phi ) = 0$ holds for all $ \phi \in \widetilde{H}^s ( \Omega)$,  where  
    \[
       B_q(u,v)\vcentcolon =\int_{\R^n}(-\Delta)^{s/2}u\,(-\Delta)^{s/2}v\,dx+\langle qu,v\rangle, \quad u,v \in H^s(\R^n).
     \]
Assume that $0$ is not a Dirichlet eigenvalue of $(-\Delta)^s  + q $, i.e. the only  solution to  
\begin{align}
\begin{cases}
(-\Delta)^s u + q  u = 0, &   \quad \text{in} \quad \Omega, \\ 
u  = 0, & \quad \text{in} \quad \Omega_e ,
\end{cases}\end{align} 
belonging to $H^s (\R^n)$ is $u \equiv 0$. 
Under these conditions, the exterior value problem 
\begin{align}\label{oo11}
\begin{cases}
(-\Delta)^s u + q  u = 0, & \quad \text{in} \quad \Omega,   \\ 
u |_{\Omega_e} = f, &  \quad \text{in} \quad \Omega_e,  
\end{cases} 
\end{align}
has a unique solution $u =u_f \in H^s (\R^n)$ for every $f \in H^s(\R^n)$. We refer to \cite{RS-fractional-calderon-low-regularity-stability,CMRU22} for more details. 
We recall the following result:
\begin{lemma}[{\cite[Lemma 2.4]{GSU20}}]  \label{Ghosh-Salo-Uhlmann-lemma_2_4}
Let $\Omega \subset \R^n$ be a bounded open set, let $0<s<1$, and assume that $q \in L^\infty (\Omega)$ is such that 
that $0$ is not a Dirichlet eigenvalue of $(-\Delta)^s  + q $. Then there is a bounded, linear, self-adjoint operator
\[
\Lambda_q : X \to X^*,  \quad \quad  X:= H^s( \R^n) / \widetilde{H}^s(\Omega), 
\]
defined by
\[
\langle \Lambda_q [f_1] , [f_2]\rangle = B_q(u_1, f_2),
\]
where $u_1\in H^s(\R^n)$ is the weak solution of \eqref{oo11} with exterior value $f_1$.
\end{lemma}

We call 
$ 
\Lambda_q : X \to X^* 
$
the exterior DN map for the fractional Schrödinger equation. We will typically use it in the situation when $f_j\in \widetilde H^s(W_j)$ for $j=1,2$, where $W_j\Subset\Omega_e$. In this case, we use the norm \begin{align}\norm{A}_{W_1, W_2} &:= \norm{A}_{\widetilde{H}^s(W_1) \to (\widetilde{H}^{s}(W_2))^*} \\
&:=\sup\{\,\abs{\langle Au_1,u_2 \rangle} \,;\, \norm{u_j}_{H^s(\R^n)}=1,u_j \in C_c^\infty(W_j)\,\}.\end{align}
For the sake of brevity, we also define $\norm{A}_{W} := \norm{A}_{W, W}$.

\subsection{The fractional conductivity equation}

The fractional gradient $\nabla^s : H^s(\R^n) \to L^2( \R^{2n} ; \R^n)$ of order $s$ is defined by
\[
\nabla^s u(x,y) = \sqrt{ \frac{C_{n,s} }{2} }  \frac{u(x) - u(y) }{ |x-y|^{n/2 + s + 1} } (x-y) ,
\]
where $C_{n,s}$ is a strictly positive constant. Let $\text{div}_s$ be the adjoint of $\nabla^s$. 
The fractional Laplacian is related to these operators via 
\[
\text{div}_s \nabla^s = (-\Delta)^s. 
\]
Let $\gamma \in L^\infty( \R^n)$ be a scalar function such that $\gamma(x)>\gamma_0>0$ for almost all $x\in\mathbb R^n$, which we call a conductivity, and define the conductivity matrix $\Theta_\gamma$ by 
\[
\Theta_\gamma(x,y) :=  \gamma^{1/2} (x) \gamma^{1/2} (y) I_{n \times n}, \quad x,y \in \R^n.
\]
We define 
\[
B_\gamma ( u,v ) = \int_{\R^{2n}} \Theta_\gamma \nabla^s u \cdot \nabla^s v dx dy, \quad u,v \in H^s(\R^n).
\]
The function $u\in H^s( \R^n)$ is said to be a weak solution of the fractional conductivity equation $\text{div}_s ( \Theta_\gamma  \nabla^s u )= 0$ in $\Omega$ if $B_\gamma ( u,\phi  ) = 0$ holds for every  $\phi \in \widetilde{H}^s ( \Omega)$.

The exterior value problem for the fractional conductivity equation reads
\begin{align}\label{cc11}
\begin{cases}
\text{div}_s ( \Theta_\gamma  \nabla^s u )= 0,   &\quad \text{in} \quad \Omega,  \\
u = f,  &\quad \text{in} \quad \Omega_e. 
\end{cases}
\end{align}
Here the exterior value condition $u = f$ in $\Omega_e$ means that $u-f \in  \widetilde{H}^s( \Omega) $.
%
Define the background deviation $m_\gamma $ by  
\[
m_\gamma  := \gamma^{1/2} - 1.
\]
We define the background deviation in this way in order to set the constant $1$ boundary condition for conductivities at infinity. It is not clear whether a similar theory would hold without assumptions on the integrability and regularity of $m_\gamma$. We recall the following result:
\begin{lemma}\label{lemm123} $($\cite[Lemma 8.10]{RZ2022unboundedFracCald}$)$
Let $\Omega$ be open set bounded in one direction and $0<s< \min (1,n/2)$.
Assume that $\gamma \in L^\infty(\R^n)$, $m_\gamma \in H^{2s, \frac{n}{2s}}(\R^n)$ and
$\gamma \geq \gamma_0$ for some constant $\gamma_0 >0$. Then there is a bounded linear operator 
\[
\Lambda_\gamma : X \to  X^*,  \quad \quad  X:= H^s( \R^n) / \widetilde{H}^s(\Omega), 
\]
satisfying 
\[
 \langle \Lambda_\gamma[f_1] , [f_2] \rangle = B_\gamma( u_1, f_2 )
\]
for all $f_1,f_2\in X$, where $u_1 $ solves \eqref{cc11} with exterior value $f_1$.
\end{lemma}

We call $\Lambda_\gamma$ the exterior DN map for the fractional conductivity equation. \\

Provided we have enough regularity (see Lemma \ref{lemma: Liouville reduction} below), the fractional conductivity equation can be rewritten as a fractional Schrödinger equation. Indeed, if $u$ is a solution to \eqref{cc11}, then the function  $v:= \gamma^{1/2} u$  solves \eqref{oo11} for the potential $ q_\gamma:= \frac{ (- \Delta)^s m_\gamma }{\gamma^{1/2}}$ and the exterior value $ \gamma^{1/2} f $. That is;
\begin{align}
\begin{cases}
(-\Delta)^s v + q_\gamma  v = 0, &  \quad  \text{in} \quad \Omega,  \\ 
v  = \gamma^{1/2} f,  &\quad   \text{in} \quad \Omega_e.
\end{cases} 
\end{align}
This technique is known as fractional Liouville reduction:

\begin{lemma}[{Liouville reduction, \cite[Lemma 3.9]{RZ2022LowReg}}]
\label{lemma: Liouville reduction}
    Let $0<s<\min(1,n/2)$. Assume that $\gamma\in L^{\infty}(\R^n)$ with conductivity matrix $\Theta_{\gamma}$ and background deviation $m_\gamma$ satisfies $\gamma(x)\geq \gamma_0>0$ and $m_\gamma\in H^{s,n/s}(\R^n)$. Let $q_\gamma:=\frac{(-\Delta)^sm_\gamma}{\gamma^{1/2}}$. 
    Then the identity
    \begin{equation}
    \label{eq: Liouville reduction}
    \begin{split}
        \langle\Theta_{\gamma}\nabla^su,\nabla^s\phi\rangle_{L^2(\R^{2n})}=&\,\langle (-\Delta)^{s/2}(\gamma^{1/2}u),(-\Delta)^{s/2}(\gamma^{1/2}\phi))\rangle_{L^2(\R^n)}\\
        &+\langle q_\gamma(\gamma^{1/2}u),(\gamma^{1/2}\phi)\rangle
    \end{split}
    \end{equation}
 holds for all $u,\phi\in H^s(\R^n)$.
\end{lemma}

\section{Partial data problem I}

In this section we prove Theorem \ref{thm: stability estimate partial I}, i.e. our stability result for the partial data problem in the case that $\gamma_1 = \gamma_2$ in $\Omega_e$. When dealing with two conductivities $\gamma_i \in L^{\infty}(\R^n)$, $i=1,2$, for the sake of simplicity we will let $m_i$ and $q_i$ be shorthand for $m_{\gamma_i} = \gamma_i^{1/2} - 1 $ and $ q_{\gamma_i} = -\frac{ (- \Delta)^s m_{\gamma_i} }{\gamma_i^{1/2}} $, respectively. We begin by recalling a known result by R\"uland and Salo.

\begin{theorem}[{\cite[Theorem~1.2]{RS-fractional-calderon-low-regularity-stability}}]\label{thm: ruland salo proposition}
Let $\Omega\subset\R^n$ be a bounded smooth domain, $0<s<1$, and $W_1,W_2 \subset \Omega_e$ be non-emtpy open sets. Assume that for some $\delta, M>0$ the potentials $q_1,q_2\in H^{\delta,\frac{n}{2s}}(\R^n)$ have the bounds
\[
\| q_j \|_{H^{\delta,\frac{n}{2s}}(\Omega)}\leq M,\quad j=1,2.
\]
Suppose also that zero is not a Dirichlet eigenvalue for the exterior value problem 
\begin{equation}\label{eq: fractional schrodinger}
    (-\Delta)^s u + q_j u = 0,\quad\text{in }\Omega
\end{equation}
with $u|_{\Omega_e}=0$, for $j=1,2$. Then it holds that
\[
\| q_1 - q_2 \|_{L^{\frac{n}{2s}}(\Omega)} \leq \omega(\| \Lambda_{q_1} - \Lambda_{q_2} \|_{\widetilde{H}^s(W_1) \to (\widetilde{H}^s(W_2))^*}),
\]
where $\Lambda_{q_j}\colon X\to X^*$ with $\vev{\Lambda_{q_j}f,g}:=B_{q_j}(u_f,g)$ is the DN map related to the exterior value problem for equation \eqref{eq: fractional schrodinger}, and $\omega$ is a modulus of continuity satisfying
\[
\omega(x) \leq C|\log x|^{-\sigma}, \quad 0< x\leq 1
\]
for two constants $C,\sigma>0$ depending only on $\Omega$, $n$, $s$, $W_1$, $W_2$, $\delta$ and $M$.
\end{theorem}

 The following lemma is a direct consequence of   \cite[Lemma 4.2]{doi:10.1137/22M1533542}. Notice that the two multiplication maps $\mathcal M^+, \mathcal M^-$ given by $\mathcal M^\pm :w \mapsto  \gamma^{\pm 1/2} w$ are inverses of each other. 
\begin{lemma}\label{lemmaeq10991}
  Let $0<s<\min(1,n/2)$, and let $U\subset\mathbb R^n$ be an open set. Assume that there exist two constants $0<\gamma_0<1$ and $\varepsilon>0$ such that $ \gamma \in L^\infty( \R^n) $ satisfies 
   \begin{align} 
&\gamma_0\leq \gamma (x) \leq \gamma_0^{-1}, \\
&m:= \gamma^{1/2}-1 \in H^{2s+\varepsilon,\frac{n}{s}}(\R^n) .
\end{align}
Then the maps $\mathcal{M}^\pm : w \mapsto \gamma^{\pm 1/2} w$  are linear homeomorphisms from $H^s(U) $ to itself. The bounds $\| \mathcal{M}^\pm \|_{H^s(U) \to H^s(U)}$ depend on $\gamma_0, C_1$. 
\end{lemma}

We wish to extend the scope of Lemma \ref{lemmaeq10991} to the following new result:

 \begin{lemma}\label{lemmaeq10993}
   If the conditions of Lemma \ref{lemmaeq10991} hold, then the multiplication maps $\mathcal{M}^\pm :  w \mapsto \gamma^{\pm 1/2} w$  are linear homeomorphisms from $\widetilde{H}^s(U) $ to itself, and the bounds $\| \mathcal{M}^\pm \|_{\widetilde{H}^s(U) \to \widetilde{H}^s(U)}$ depend on $\gamma_0$ and $C_1$.  
\end{lemma}

In order to prove Lemma \ref{lemmaeq10993}, we first show the following auxiliary Lemma:
   
   \begin{lemma}\label{lemmaeq10992}
  If the conditions of Lemma \ref{lemmaeq10991} hold, then $\gamma^{\pm 1/2} \psi \in \widetilde{H}^s(U)$ for every $\psi \in C^\infty_c(U)$. 
\end{lemma}
   
   \begin{proof}

   Let $\chi \in C_c^\infty(U)$ be such that it equals 1 on the support of $\psi$. By Lemma \ref{lemmaeq10991}, $ \gamma^{\pm 1/2} \chi \in H^s(\R^n)$. 
   By the Meyers–Serrin theorem, we can find two sequences $\{\varphi^\pm_k\}_{k\in\mathbb N} \subset C^\infty (\R^n) \cap H^s(\R^n)$ such that  $\varphi^\pm_k \rightarrow \gamma^{\pm 1/2} \chi$ in $H^s(\R^n)$ as $k \to \infty$. 
   Then $\varphi^\pm_k \psi  \in C^\infty_c( U)$ and 
   \begin{align}\label{237777}
   \| \varphi^\pm_k \psi -  \gamma^{\pm 1/2} \psi \|_{H^s(\R^n)} &=    \| \varphi^\pm_k \psi -  \gamma^{\pm 1/2} \chi \psi \|_{H^s(\R^n)}\\
   &\leq    \| \varphi^\pm_k  -  \gamma^{\pm 1/2} \chi  \|_{H^s(\R^n)} \| \psi \|_{H^s(\R^n)} .
   \end{align}
   As $ \psi \in C^\infty_c (U)$, the Fourier transform $\widehat\psi$ is rapidly decaying by Paley-Wiener theorem. Hence $ \| \psi \|_{H^s(\R^n)} = \| \langle \xi \rangle^s   \hat\psi  \|_{L^2(\R^n)} < \infty$. 
   This, together with the convergence $\varphi^\pm_k \to \gamma^{\pm 1/2}  \chi $, implies that the right-hand side in \eqref{237777} vanishes as $k \to \infty$. 
   Therefore, the sequence $\{\varphi^\pm_k \psi\}_{k\in\mathbb N}  \subset C^\infty_c(U)$ converges to $  \gamma^{\pm 1/2} \psi$ in $H^s(\R^n)$. Hence, $\gamma^{\pm 1/2} \psi \in  \widetilde{H}^s(U)$.  
   \end{proof}

   Using the above Lemma \ref{lemmaeq10992}, we now prove Lemma \ref{lemmaeq10993}.
   
\begin{proof}[Proof of Lemma \ref{lemmaeq10993}]
   
   It suffices to prove that for each $u \in  \widetilde{H}^s(U)$ the image $\gamma^{\pm 1/2} u$ lies in  $ \widetilde{H}^s(U)$. 
	The continuity then follows by Lemma \ref{lemmaeq10991}. 
   Fix arbitrary $w \in  \widetilde{H}^s(U)$ and test functions $\phi_l \in C_c^\infty (U)$, $l=1,2,3,\dots$ such that   $\phi_l \to w$ in $H^s(\R^n)$. 
   We have $\gamma^{\pm 1/2 } \phi_l \in \widetilde{H}^s(U)$ for every $l=1,2,3,\dots,$ by Lemma \ref{lemmaeq10992}. 
   Moreover, by Lemma \ref{lemmaeq10991} we have $\gamma^{\pm 1/2 } \phi_l \to \gamma^{\pm 1/2 } w$ in $H^s(\R^n)$. 
   Hence, $\gamma^{\pm 1/2 } w $ lies in the closure of $ \widetilde{H}^s(U)$ in $H^s(\R^n)$. 
 Because $ \widetilde{H}^s(U)$ is closed in $H^s(\R^n)$ by definition, we conclude  $\gamma^{\pm 1/2 } w  \in \widetilde{H}^s(U)$. 
\end{proof}

\begin{remark} Lemma \ref{lemmaeq10993} could also be proved using the Gagliardo--Nirenberg inequality and sharper multiplication estimates following the work \cite[the first item on p. 4, Lemma 3.4 (7), and Corollary 3.6]{RZ2022LowReg}. We have given this direct proof here to keep our current work self-contained.
\end{remark}

The next Lemma relates the DN map of the original problem for the fractional conductivity equation to the DN map of the Schr\"odinger problem obtained after the fractional Liouville reduction. 

\begin{lemma}\label{step1} 
Assume that the conditions of Theorem \ref{thm: stability estimate partial I} hold. Then, 
\[
    \norm{\Lambda_{q_1}-\Lambda_{q_2}}_{\widetilde{H}^s (W_1) \to (\widetilde{H}^s (W_2))^*} \leq c\norm{\Lambda_{\gamma_1}-\Lambda_{\gamma_2}}_{\widetilde{H}^s (W_1) \to (\widetilde{H}^s (W_2))^*}. 
\]
where the constant $c$ depends on $\gamma_0$, $\Omega$, $W_1$, $W_2$, $n$, $s$ and $C_1$.
\end{lemma}

\begin{proof}
    Denote by $v^{i}_f\in H^s(\R^n)$, $i=1,2$ the unique solution to the fractional Schr\"odinger equation \eqref{oo11} with exterior value $f \in H^s(W_1)$ and potential $ q_i := -\frac{(-\Delta)^sm_i}{\gamma_i^{1/2}}$.
    Using Lemma~\ref{lemma: Liouville reduction},  we deduce for $i=1,2$ that the following identity holds for all $g\in H^s(W_2)$:
    \[
        \langle \Lambda_{q_i} f,g\rangle= B_{q_i}(v_f^{i},g)=B_{\gamma_i}(\gamma_i^{-1/2}v_f^{i}, \gamma_{i}^{-1/2} g)=\langle \Lambda_{\gamma_i} (\gamma_i^{-1/2}f), \gamma_i^{-1/2}g \rangle.
    \]
   Hence, 
    \[
    \begin{split}
        &\langle (\Lambda_{q_1}-\Lambda_{q_2}) f, g\rangle \\ &= \langle \Lambda_{\gamma_1} (\gamma_1^{-1/2}f), \gamma_1^{-1/2}g \rangle -\langle \Lambda_{\gamma_2}(\gamma_2^{-1/2}f),\gamma_2^{-1/2}g \rangle \\
        &=\,\langle \Lambda_{\gamma_1} (\gamma_1^{-1/2}f),(\gamma_1^{-1/2}-\gamma_2^{-1/2})g \rangle  +\langle \Lambda_{\gamma_2} (\gamma_1^{-1/2}-\gamma_2^{-1/2})f, \gamma_2^{-1/2}g \rangle\\
        &\quad \,\, +\langle (\Lambda_{\gamma_1}-\Lambda_{\gamma_2}) (\gamma_1^{-1/2}f), \gamma_2^{-1/2}g\rangle.\\
    \end{split}
    \]
   As $\gamma_1 = \gamma_2 $ in $\Omega_e$,   we have that both $   (\gamma_1^{-1/2}-\gamma_2^{-1/2})g$ and $(\gamma_1^{-1/2}-\gamma_2^{-1/2})f $ are zero in $\Omega_e$. 
    Thus,  the terms $    \langle \Lambda_{\gamma_1} (\gamma_1^{-1/2}f) , (\gamma_1^{-1/2}-\gamma_2^{-1/2})g \rangle   $ and $\langle \Lambda_{\gamma_2} (\gamma_1^{-1/2}-\gamma_2^{-1/2})f , \gamma_2^{-1/2}g  \rangle$ above vanish and 
    we obtain 
    \begin{equation}\label{eq: identity}
             \langle (\Lambda_{q_1}-\Lambda_{q_2}) f,g\rangle = \langle (\Lambda_{\gamma_1}-\Lambda_{\gamma_2}) (\gamma_1^{-1/2}f),\gamma_2^{-1/2}g\rangle.
    \end{equation}
    Using Lemma \ref{lemmaeq10993} and \eqref{eq: identity}, we obtain the estimate 
    \begin{equation}\label{eq:Schrodinger-to-conductivity-DN-maps-estimate}
    \norm{\Lambda_{q_1}-\Lambda_{q_2}}_{\widetilde{H}^s (W_1) \to (\widetilde{H}^s (W_2))^*} \leq c\norm{\Lambda_{\gamma_1}-\Lambda_{\gamma_2}}_{\widetilde{H}^s (W_1) \to (\widetilde{H}^s (W_2))^*},
    \end{equation}
	where the constant $c>0$  depends on $\gamma_0$, $\Omega$, $W_1$, $W_2$, $n$, $s$ and $C_1$.
	\end{proof}

In the next proposition we relate the DN map for the original fractional conductivity problem directly to the transformed Schr\"odinger potentials: 

\begin{proposition}\label{step2}
Assume that the conditions of Theorem \ref{thm: stability estimate partial I} hold. Then, 
\[
     \| q_1 - q_2 \|_{L^{\frac{n}{2s}}(\Omega)} \leq \omega(\| \Lambda_{\gamma_1} - \Lambda_{\gamma_2} \|_{\widetilde{H}^s (W_1) \to (\widetilde{H}^s (W_2))^*}),
\]
 where $\omega$ is a logarithmic modulus of continuity satisfying
    \[
     \omega(t) \leq C|\log t|^{-\sigma},\quad\text{for}\quad 0 < t\leq \lambda,
    \]
    for three constants $\sigma,\lambda,C>0$ depending on $s,\varepsilon,n,\Omega, W_1,W_2$ and $\gamma_0$. 
\end{proposition}		
\begin{proof}
    It suffices to show that the conditions of Theorem \ref{thm: ruland salo proposition} are met, since the claim will then follow by Lemma \ref{step1}. 
   To this end, we will show that there exist two constants $\delta, M > 0$ such that 
    \begin{equation}\label{qequation}
\| q_j \|_{H^{\delta,\frac{n}{2s}}(\Omega)}\leq M,\quad j=1,2.
\end{equation}
We begin by writing $q_j$ in the form 
\begin{align}
q_j &=- \frac{1}{\gamma^{1/2}_j } (- \Delta)^s m_j =- \left( 1 - \frac{m_j}{m_j + 1} \right) (- \Delta)^s m_j  \\&= -  (- \Delta)^s m_j   +   \frac{m_j}{m_j + 1}  (- \Delta)^s m_j.  \label{kaksitermia}
\end{align}
Set $\delta := \min(s,\varepsilon) $. 
Notice that  $m_j  \in H^{2s+\varepsilon, n/s} (\R^n)$ by assumption. Hence, the term $ -  (- \Delta)^s m_j$ above belongs to $H^{\varepsilon, \frac{n}{s} } (\R^n) \subseteq H^{\delta, \frac{n}{s} } (\R^n)  $ which is embedded continuously into a subspace of $ H^{\delta, \frac{n}{s} } (U)$ via the quotient map. 
We recall that $H^{s, p } (\Omega) $ coincides with the Triebel-Lizorkin space $ F_{p,2}^s ( \Omega)  $ for every $p>1$ and $s >0$. By \cite[Chapter 2.4.4, Theorem 1]{RunstSickel+1996}, 
$F_{\frac{n}{s},2}^\delta ( \Omega)  \hookrightarrow F_{\frac{n}{2s},2}^\delta ( \Omega)$. Hence, $H^{\delta, \frac{n}{s} } (\Omega)  \hookrightarrow H^{\delta, \frac{n}{2s} } (\Omega) $ and we conclude  
\[
  \|   (- \Delta)^s m_j  \|_{H^{\delta, \frac{n}{2s}} (\Omega)} < \infty. 
  \]
  Let us now focus on the last term in \eqref{kaksitermia}. 
First of all,  $ \frac{m_j}{m_j + 1} \in H^{s, \frac{n}{s} } (\R^n)$,  as shown in the proof of \cite[Corollary 3.6]{RZ2022LowReg}.
In particular, we have  $\frac{m_j}{m_j + 1} \in H^{\delta, \frac{n}{s} } (\R^n)$ and  $\frac{m_j}{m_j + 1} \in L^{\frac{n}{s}} (\R^n)$. Moreover, 
 $ (- \Delta)^s m_j \in H^{\delta, \frac{n}{s} } (\R^n) \subset L^{\frac{n}{s}} (\R^n) $, as shown above. 
By the Kato-Ponce inequality (see the original work \cite{https://doi.org/10.1002/cpa.3160410704} or \cite[Proposition 2.2]{RZ2022LowReg})
\begin{align}
\left\|  \frac{m_j}{m_j + 1}  (- \Delta)^s m_j \right\|_{ H^{\delta, \frac{n}{2s}} (\R^n) }  
 \lesssim \left\|  \frac{m_j}{m_j + 1}  \right\|_{ H^{\delta , \frac{n}{s}} (\R^n) } \left\|  (- \Delta)^s m_j \right\|_{ L^{ \frac{n}{s}} (\R^n) } \\
+ \left\|  \frac{m_j}{m_j + 1}  \right\|_{ L^{\frac{n}{s}} (\R^n) } \left\|  (- \Delta)^s m_j \right\|_{ H^{\delta, \frac{n}{s}} (\R^n) }. 
\end{align}
As argued above, the right hand side is finite. 
Thus, the term $\frac{m_j}{m_j + 1}  (- \Delta)^s m_j$ lies in $H^{\delta, \frac{n}{2s}} (\R^n)$. 
In conclusion, 
\[
\| q_j \|_{ H^{\delta, \frac{n}{2s}} (\Omega) } < \infty .  
\] 
and  $\| q_j \|_{H^{\delta,\frac{n}{2s}}(\Omega)}\leq M< \infty $ holds for 
$
M :=  \max_{j=1,2} \| q_j  \|_{H^{\delta, \frac{n}{2s}} (\Omega)}.
$
    \end{proof}

    Finally, we derive an estimate connecting  $\| \gamma_1^{1/2} - \gamma_2^{1/2} \|_{H^{s}(\Omega)} $ to  $   \| q_1 - q_2 \|_{L^{\frac{n}{2s}}(\Omega)}$:
    \begin{lemma}\label{step3}
    Assume that the conditions of Theorem \ref{thm: stability estimate partial I} hold. Then, 
    \[
    \|\gamma_1-\gamma_2\|_{H^{s}(\Omega)} \leq C' \|\gamma_1^{1/2}-\gamma_2^{1/2}\|_{H^{s}(\Omega)}   \leq C \norm{q_1-q_2}_{L^{\frac{n}{2s}}(\Omega)}.
    \]
    where the constant $C,C' >0$ depend on $\gamma_0, C_1$ and $\Omega$.  
    \end{lemma}
    
\begin{proof} 
We may write
\begin{equation}
    \gamma_1-\gamma_2 = (\gamma_1^{1/2}+\gamma_2^{1/2})(\gamma_1^{1/2}-\gamma_2^{1/2}).
\end{equation} We then directly conclude from Lemma \ref{lemmaeq10991} that 
$$\|\gamma_1-\gamma_2\|_{H^{s}(\Omega)} \leq C' \|\gamma_1^{1/2}-\gamma_2^{1/2}\|_{H^{s}(\Omega)}$$
where $C'>0$ depend only on $\gamma_0$ and $C_1$. This proves the first inequality.

By assumptions, the support of $\gamma_1^{1/2}-\gamma_2^{1/2}$ lies in the compact set $\overline\Omega$. 
In particular, there is a smooth cutoff function $\chi \in C_c^\infty( \R^n)$ such that 
\[
\gamma_1^{1/2}-\gamma_2^{1/2} = (\gamma_1^{1/2}  -\gamma_2^{1/2}) \chi  = \gamma_1^{1/2} \chi   -\gamma_2^{1/2} \chi. 
\]
By Lemma \ref{lemmaeq10991}, 
\begin{equation}\label{kkao3}
\gamma_j^{1/2} \chi \in H^s(\R^n), \quad j=1,2. 
\end{equation} 
The function $v :=\gamma_1^{-1/2}(\gamma_1^{1/2}-\gamma_2^{1/2})$ belongs to $H^s(\R^n)$ by Lemma \ref{lemmaeq10991} and \eqref{kkao3}. 
Denote $F:=\gamma_{1}^{1/2}\gamma_{2}^{1/2}(q_1-q_2)  $. 
By \cite[Lemma 4.4]{doi:10.1137/22M1533542}, $v$ solves 
    \begin{align}
       \text{div}_s ( \Theta_{\gamma_1} \nabla^s  v )  = F  \quad &
        \text{in $\Omega$},\\
        v =0\quad   &  \text{in $\Omega_e$} ,
    \end{align} 
   from which we get an elliptic estimate
   \[
   \| v \|_{H^s(\Omega)} \leq  C_4 \| F\|_{H^{-s}(\Omega)},
   \]
   where $C_4$ depends on $\gamma_0$. 
   Because $H^{-s}(\Omega) = (\widetilde{H}^{s}(\Omega))^*$, by duality and Lemma \ref{lemmaeq10991} we obtain 
   \begin{equation}
       \norm{F}_{H^{-s}(\Omega)} \leq C_3 \norm{q_1-q_2}_{H^{-s}(\Omega)} ,
   \end{equation}
   where $C_3$ depends on $\gamma_0, C_1$. 
    Combining the last two inequalities,  we obtain
   \begin{equation}\label{vequation}
      \| v \|_{H^s(\Omega)} \leq C_2 \norm{q_1-q_2}_{H^{-s}(\Omega)} ,
   \end{equation} 
   where $C_2$ depends on $\gamma_0, C_1$. 
By H\"older's inequality and the Sobolev embedding theorem, for all $w\in \widetilde{H}^s( \Omega)$ it holds that  
\[
\| w \|_{L^{\frac{n}{n-2s} } (\Omega)} \leq C_\Omega \| w \|_{L^{{\frac{2n}{n-2s}  }} (\Omega)} \leq C_\Omega \| w \|_{ \widetilde{H}^s( \Omega) }.
\]
Passing to the associated dual spaces, for all $\psi \in L^{\frac{n}{2s}  }(\Omega)$ we have 
\[
 \| \psi \|_{ H^{-s}( \Omega) } =  \| \psi \|_{ (\widetilde{H}^s( \Omega))^* } \leq C'_\Omega  \| \psi \|_{ \big( L^{{\frac{n}{n-2s} }  } (\Omega) \big)^*} = C'_\Omega \| \psi \|_{L^{{\frac{n}{2s}  }} (\Omega)}.
 \]
In particular, we deduce that 
   \[
 \norm{q_1-q_2}_{H^{-s}(\Omega)}  \leq C'_\Omega \norm{q_1-q_2}_{L^{\frac{n}{2s}}(\Omega)}, 
   \]
and combining this inequality with \eqref{vequation} yields
\begin{equation}\label{wie}
  \| v \|_{H^s(\Omega)} \leq C_2 C'_\Omega  \norm{q_1-q_2}_{L^{\frac{n}{2s}}(\Omega)}.
\end{equation}
However, becuse  
$\gamma_1^{1/2} - \gamma_2^{1/2} =  \gamma_1^{1/2}\chi -  \gamma_2^{1/2}  \chi \in H^s (\R^n)$, we have
\[
 \| \gamma_1^{1/2} - \gamma_2^{1/2}  \|_{H^s(\Omega)}  \leq C'  \| \gamma_1^{-1/2} ( \gamma_1^{1/2} - \gamma_2^{1/2}  ) \|_{H^s(\Omega)} = C'  \| v \|_{H^s(\Omega)} 
\]
by Lemma \ref{lemmaeq10991}, where $C'$ depends on $\gamma_0$ and $C_1$. 
Combining this with \eqref{wie} completes the proof. 
\end{proof}

With the above discussion in mind, it is now easy to prove Theorem \ref{thm: stability estimate partial I}.
   
\begin{proof}[Proof of Theorem \ref{thm: stability estimate partial I}]
The claim follows by combining Proposition \ref{step2} and Lemma \ref{step3}. 
\end{proof}

\section{Partial data problem II}
We now turn to the proof of Theorem \ref{thm: stability estimate partial II}, for which we will apply the two following propositions:

\begin{proposition}[{\cite[Proposition 5.1]{doi:10.1137/22M1533542}}]
\label{thm: partial data reduction}
    Let $0<s<\min(1,n/2)$, $s/n < \theta_0 < 1$ and $\varepsilon > 0$.
    Let $\Omega\subset \R^n$ be a non-empty open set bounded in one direction. Suppose that the conductivities $\gamma_1,\gamma_2\in L^{\infty}(\R^n)$ with background deviations $m_1,m_2$ and potentials $q_1,q_2$ satisfy the following conditions:
    \begin{enumerate}[(i)]
        \item\label{item: partial assumption 1} $\gamma_0\leq \gamma_1(x),\gamma_2(x)\leq \gamma_0^{-1}$ for some $0<\gamma_0<1$,
        \item\label{item: partial assumption 3}
       there exists $C_0>0$ such that
        \begin{equation}
        \label{eq: partial difference bound}
            \|m_i\|_{H^{\frac{2s+\varepsilon}{\theta_0},\frac{\theta_0 n}{s}}(\R^n)}\leq C_0
        \end{equation}
         for $i=1,2$.
    \end{enumerate} 
    Let $W_1,W_2,W \Subset \Omega_e$ be non-empty open sets such that $W_1 \cup W_2 \Subset W$. 
    There exists a constant $C>0$ depending only on $s,\varepsilon,n,\Omega, C_0, \theta_0, W_1, W_2, W$ and $\gamma_0$ such that
    \begin{equation}
    \label{eq: partial reduction}
\|\Lambda_{q_1}-\Lambda_{q_2}\|_{W_1,W_2}\leq C(\|\Lambda_{\gamma_1}-\Lambda_{\gamma_2}\|_{W}+\|\Lambda_{\gamma_1}-\Lambda_{\gamma_2}\|_{W}^{\frac{1-\theta_0}{2}}).
    \end{equation}
\end{proposition}

\begin{proposition}[{\cite[Proposition 6.1]{GRSU-fractional-calderon-single-measurement}}]\label{prop:GRSU-single-quantitative-ucp} Let $n \geq 1$, $0 < s < 1$ and $s' < s$. Let $\Omega, W \subset \R^n$ be open bounded Lipschitz domains with $\overline{\Omega}\cap \overline{W} = \emptyset$. Then there exist constants $C, \sigma > 0$ only depending on $\Omega, W, n, s, s'$ such that for any $E > 0$ and $v \in H_{\overline{\Omega}}^s$ with $\norm{v}_{H_{\overline{\Omega}}^s} \leq E$ it holds that
\begin{equation}
\norm{v}_{H_{\overline{\Omega}}^{s'}} \leq CE \log\left(\frac{CE}{\norm{(-\Delta)^s v}_{H^{-s}(W)}}\right)^{-\sigma}.
\end{equation}
\end{proposition}

\begin{proof}[Proof of Theorem \ref{thm: stability estimate partial II}]
From here onwards, we shall write $a \lesssim b$ if $a \leq c b $ holds for some constant $c>0$ that may depend on $s,p,W,\Sigma,\varepsilon,n,\Omega, C_1$,$\gamma_0$. 
Let $W' \subset W$ be a nonempty, bounded, Lipschitz domain such that $\overline{W}' \subset W$. 

By Proposition \ref{prop:GRSU-single-quantitative-ucp},  we have for $v:= \gamma_1^{1/2}-\gamma_2^{1/2}$ that 
\begin{equation}
\norm{v}_{H_{\overline{\Sigma'}}^{s'}} \leq CE \log\left(\frac{CE}{\norm{(-\Delta)^sv}_{H^{-s}(W')}}\right)^{-\sigma},
\end{equation}
where $\Sigma'\supset\Sigma$ is open, bounded, Lipschitz and such that $\overline{\Sigma'\cap W} =\emptyset$. Here $E>0$ is chosen such that $\norm{v}_{H_{\overline{\Sigma'}}^{s}} <E$, which is allowed by assumption (iii). By the support assumptions and Sobolev embedding theorem, we then have
\begin{equation}
\|v\|_{L^{\frac{2n}{n-2s'}}(\R^n)} =  \|v\|_{L^{\frac{2n}{n-2s'}}(\Sigma')} \lesssim \norm{v}_{H_{\overline{\Sigma'}}^{s'}} \leq CE \log\left(\frac{CE}{\norm{(-\Delta)^sv}_{H^{-s}(W')}}\right)^{-\sigma},
\end{equation}
and since in $W'$ it holds that $\gamma_1=\gamma_2$, we can compute \begin{equation}\begin{split}
\norm{(-\Delta)^s v}_{H^{-s}(W')} & = \norm{\gamma_1^{1/2} q_1 - \gamma_2^{1/2}q_2}_{H^{-s}(W')} 
\\ & \leq
\norm{(\gamma_1^{1/2} - \gamma_2^{1/2})q_1}_{H^{-s}(W')} + \norm{\gamma_2^{1/2} (q_1 - q_2)}_{H^{-s}(W')}
\\ & \lesssim
\norm{(\gamma_1^{1/2} - \gamma_2^{1/2})q_1}_{L^{2}(W')} + \norm{q_1 - q_2}_{H^{-s}(W')}
\\ & =
\norm{q_1 - q_2}_{H^{-s}(W')},
\end{split}\end{equation} 
where Lemma \ref{lemmaeq10993} was used in the last inequality. 
This  leaves us with 
$$\|v\|_{L^{p}(\R^n)} \lesssim CE \log\left(\frac{CE}{\norm{q_1-q_2}_{H^{-s}(W')}}\right)^{-\sigma},$$
where $p = \frac{2n}{n-2s'}$. The inequality $s'<s$ gives the upper bound $p < \frac{2n}{n-2s}$. The lower bound $1\leq p$ is needed for $\| \cdot \|_{L^p}$ to be a norm. We may use an argument similar to the proof of the first inequality in Lemma \ref{step3} to pass from the difference of square roots $v$ to the difference of conductivities $\gamma_1-\gamma_2$. To complete the proof, we thus only need to show that there exists some logarithmic modulus of continuity $\omega$ such that
\begin{equation}
\norm{q_1-q_2}_{H^{-s}(W')} \leq \omega(\norm{\Lambda_{\gamma_1}-\Lambda_{\gamma_2}}_W).
\end{equation}

Let $u_i\in H^s(\R^n)$ solve the problem $((-\Delta)^s+q_i)u_i=0$ in $\Omega$ with $u_i|_{\Omega_e}=f_i$, where $f_i \in C_c^\infty(W)$ and $i=1,2$. By the Alessandrini identity,
\begin{equation}\label{eq: splitting formula}\begin{split}
\int_{W} (q_1-q_2)f_1f_2dx & = \int_{\R^n}(q_1-q_2)u_1u_2dx - \int_\Omega (q_1-q_2)u_1u_2dx
\\ & = \ip{(\Lambda_{q_1}-\Lambda_{q_2})f_1}{f_2} - \int_\Omega (q_1-q_2)u_1u_2dx,
\end{split}\end{equation}
and thus we can estimate
\begin{align}
\abs{\int_{W} (q_1-q_2)f_1f_2dx} &\leq \norm{\Lambda_{q_1}-\Lambda_{q_2}}_{\widetilde{H}^s(W) \to (\widetilde{H}^s(W))^*} \norm{f_1}_{H^s(\R^n)}\norm{f_2}_{H^s(\R^n)} \\&\quad\,\,+ \abs{\int_\Omega (q_1-q_2)u_1u_2dx}.
\end{align}
Repeating the argument in the proof of Proposition \ref{step2}, we deduce that $\| q_j \|_{H^{\delta, n/2s} (\Omega) } $ is finite for some $\delta >0$. 
For the last term above, we obtain  using duality and the elliptic estimate that 
\begin{equation}\begin{split}
    \abs{\int_\Omega (q_1-q_2)u_1u_2dx} 
     & \lesssim
    \norm{q_1-q_2}_{L^{\frac{n}{2s}}(\Omega)} \norm{f_1}_{H^s(\R^n)}\norm{f_2}_{H^s(\R^n)}.
\end{split}
\end{equation} 
Applying Theorem \ref{thm: ruland salo proposition} and Proposition \ref{thm: partial data reduction}, we get
\begin{equation}\begin{split}
\abs{\int_{W} (q_1-q_2)f_1f_2dx} & \lesssim \omega_1(\| \Lambda_{q_1} - \Lambda_{q_2} \|_{W})\norm{f_1}_{H^s(\R^n)}\norm{f_2}_{H^s(\R^n)}
\\ & \leq
\omega_1(\omega_2(\| \Lambda_{\gamma_1} - \Lambda_{\gamma_2} \|_{W}))\norm{f_1}_{H^s(\R^n)}\norm{f_2}_{H^s(\R^n)},
    \end{split}\end{equation} 
where $\omega_1(t) := t+\omega_0(t)$, $\omega_2(t):= t+t^{\frac{1-\theta_0}{2}}$, and $\omega_0$ is the logarithmic modulus of continuity given by Theorem \ref{thm: ruland salo proposition}. Consequently, 
\begin{equation}
\norm{q_1-q_2}_{M( \widetilde{H}^s (W) \to H^{-s} (W) )} \leq \widetilde\omega(\norm{\Lambda_{\gamma_1}-\Lambda_{\gamma_2}}_W) \label{1992}
\end{equation}
for some logarithmic modulus of continuity $\widetilde\omega$. Here the Sobolev multiplier norm is defined as (see \cite[Ch. 3]{Mazya})
\[
\norm{f}_{M( \widetilde H^s(W) \to H^{-s} (W) )}  := \sup \{ | \langle f , u_1 u_2 \rangle  | :  \| u_j \|_{\widetilde H^s(W)} = 1, \ j=1,2  \}.
\]
Fix $\chi_0 \in  C^\infty_c (W)$ such that $\chi_0 |_{W'}=1 $ and define $\chi = \| \chi_0 \|_{H^s(\R^n) }^{-1} \chi_0$. 
Then, $\| \chi \|_{H^s(\R^n)} = 1 $ and $\phi = \chi_0 \phi = \| \chi_0 \|_{H^s(\R^n)} \chi \phi$  for every $\phi \in C_c^\infty(W')$. Also
\[
| \langle q_1 - q_2 , \phi \rangle | = \| \chi_0 \|_{H^s(\R^n)}   | \langle q_1 - q_2 , \chi   \phi   \rangle |   \lesssim  | \langle q_1 - q_2 , \chi   \phi   \rangle |, 
\]
which implies that 
\[
\| q_1 - q_2 \|_{ (\widetilde{H}^{s}(W'))^* } \lesssim \| q_1 - q_2 \|_{M( \widetilde{H}^s (W) \to H^{-s} (W) )}  . 
\] 
Alternatively, this could be derived using \cite[Lemma 2.2]{RS-fractional-calderon-low-regularity-stability}. 
As $(\widetilde{H}^{s}(W'))^* = H^{-s} (W')$, we obtain 
\[
\| q_1 - q_2 \|_{ H^{-s}(W') } \lesssim \| q_1 - q_2 \|_{M( \widetilde{H}^s (W) \to H^{-s} (W) )}  . 
\] 
Combining the estimate with \eqref{1992} then finishes the proof. 
\end{proof}

\bibliography{refs} 

\begin{thebibliography}{DSFKSU09}

\bibitem[Ale88]{Alessandrini-stability}
Giovanni Alessandrini.
\newblock Stable determination of conductivity by boundary measurements.
\newblock {\em Appl. Anal.}, 27(1-3):153--172, 1988.

\bibitem[AP06]{AP06}
Kari Astala and Lassi P\"{a}iv\"{a}rinta.
\newblock Calder\'{o}n's inverse conductivity problem in the plane.
\newblock {\em Ann. of Math. (2)}, 163(1):265--299, 2006.

\bibitem[APL05]{AstalaPaivarintaLassasAnisotropic}
Kari Astala, Lassi P\"{a}iv\"{a}rinta, and Matti Lassas.
\newblock Calder\'{o}n's inverse problem for anisotropic conductivity in the
  plane.
\newblock {\em Comm. Partial Differential Equations}, 30(1-3):207--224, 2005.

\bibitem[AV05]{AlessandriniVessellaLipschitz}
Giovanni Alessandrini and Sergio Vessella.
\newblock Lipschitz stability for the inverse conductivity problem.
\newblock {\em Adv. in Appl. Math.}, 35(2):207--241, 2005.

\bibitem[Cal80]{C80}
Alberto-P. Calder\'{o}n.
\newblock On an inverse boundary value problem.
\newblock In {\em Seminar on {N}umerical {A}nalysis and its {A}pplications to
  {C}ontinuum {P}hysics ({R}io de {J}aneiro, 1980)}, pages 65--73. Soc. Brasil.
  Mat., Rio de Janeiro, 1980.

\bibitem[CDR16]{CaroDSFerreiraRuiz}
Pedro Caro, David {Dos Santos Ferreira}, and Alberto Ruiz.
\newblock Stability estimates for the calderón problem with partial data.
\newblock {\em Journal of Differential Equations}, 260(3):2457--2489, 2016.

\bibitem[CGRU23]{CGRU23}
Giovanni Covi, Tuhin Ghosh, Angkana Rüland, and Gunther Uhlmann.
\newblock A reduction of the fractional {C}alder\'on problem to the local
  {C}alder\'on problem by means of the {C}affarelli-{S}ilvestre extension,
  2023.
\newblock arXiv:2305.04227.

\bibitem[CLR20]{CLR18-frac-cald-drift}
Mihajlo Ceki\'{c}, Yi-Hsuan Lin, and Angkana R\"{u}land.
\newblock The {C}alder\'{o}n problem for the fractional {S}chr\"{o}dinger
  equation with drift.
\newblock {\em Calc. Var. Partial Differential Equations}, 59(3):Paper No. 91,
  46, 2020.

\bibitem[CMR21]{CMR20}
Giovanni Covi, Keijo M\"{o}nkk\"{o}nen, and Jesse Railo.
\newblock Unique continuation property and {P}oincar\'{e} inequality for higher
  order fractional {L}aplacians with applications in inverse problems.
\newblock {\em Inverse Probl. Imaging}, 15(4):641--681, 2021.

\bibitem[CMRU22]{CMRU22}
Giovanni Covi, Keijo M\"{o}nkk\"{o}nen, Jesse Railo, and Gunther Uhlmann.
\newblock The higher order fractional {C}alder\'{o}n problem for linear local
  operators: {U}niqueness.
\newblock {\em Adv. Math.}, 399:Paper No. 108246, 2022.

\bibitem[Cov20]{covi2019inverse-frac-cond}
Giovanni Covi.
\newblock Inverse problems for a fractional conductivity equation.
\newblock {\em Nonlinear Anal.}, 193:111418, 18, 2020.

\bibitem[Cov22]{C21}
Giovanni Covi.
\newblock Uniqueness for the fractional {C}alder\'on problem with quasilocal
  perturbations.
\newblock {\em SIAM J. Math. Anal.}, 54(6):6136--6163, 2022.

\bibitem[Cov24]{covi-survey}
Giovanni Covi.
\newblock The inverse problem for the fractional conductivity equation: a
  survey, 2024.
\newblock arXiv:2408.14200.

\bibitem[CRTZ24]{doi:10.1137/22M1533542}
Giovanni Covi, Jesse Railo, Teemu Tyni, and Philipp Zimmermann.
\newblock Stability estimates for the inverse fractional conductivity problem.
\newblock {\em SIAM Journal on Mathematical Analysis}, 56(2):2456--2487, 2024.

\bibitem[CRZ22]{RGZ2022GlobalUniqueness}
Giovanni Covi, Jesse Railo, and Philipp Zimmermann.
\newblock The global inverse fractional conductivity problem, 2022.
\newblock arXiv:2204.04325.

\bibitem[CS14]{CaroSaloAnisotropic}
Pedro Caro and Mikko Salo.
\newblock Stability of the calderón problem in admissible geometries.
\newblock {\em Inverse Problems and Imaging}, 8(4):939--957, 2014.

\bibitem[CWHM17]{MR3620760}
S.~N. Chandler-Wilde, D.~P. Hewett, and A.~Moiola.
\newblock Sobolev spaces on non-{L}ipschitz subsets of {$\mathbb{R}^n$} with
  application to boundary integral equations on fractal screens.
\newblock {\em Integral Equations Operator Theory}, 87(2):179--224, 2017.

\bibitem[DSFKSU09]{FKS-Anisotropic}
David Dos Santos~Ferreira, Carlos~E. Kenig, Mikko Salo, and Gunther Uhlmann.
\newblock Limiting {C}arleman weights and anisotropic inverse problems.
\newblock {\em Invent. Math.}, 178(1):119--171, 2009.

\bibitem[GLX17]{GLX-calderon-nonlocal-elliptic-operators}
Tuhin Ghosh, Yi-Hsuan Lin, and Jingni Xiao.
\newblock The {C}alder\'{o}n problem for variable coefficients nonlocal
  elliptic operators.
\newblock {\em Comm. Partial Differential Equations}, 42(12):1923--1961, 2017.

\bibitem[GRSU20]{GRSU-fractional-calderon-single-measurement}
Tuhin Ghosh, Angkana R\"{u}land, Mikko Salo, and Gunther Uhlmann.
\newblock Uniqueness and reconstruction for the fractional {C}alder\'{o}n
  problem with a single measurement.
\newblock {\em J. Funct. Anal.}, 279(1):108505, 42, 2020.

\bibitem[GSU20]{GSU20}
Tuhin Ghosh, Mikko Salo, and Gunther Uhlmann.
\newblock The {C}alder\'{o}n problem for the fractional {S}chr\"{o}dinger
  equation.
\newblock {\em Anal. PDE}, 13(2):455--475, 2020.

\bibitem[GU21]{ghosh2021calderon}
Tuhin Ghosh and Gunther Uhlmann.
\newblock The {C}alderón problem for nonlocal operators, 2021.
\newblock arXiv:2110.09265.

\bibitem[KP88]{https://doi.org/10.1002/cpa.3160410704}
Tosio Kato and Gustavo Ponce.
\newblock Commutator estimates and the euler and navier-stokes equations.
\newblock {\em Communications on Pure and Applied Mathematics}, 41(7):891--907,
  1988.

\bibitem[KRS21]{KochRulandSaloInstability}
Herbert Koch, Angkana R\"{u}land, and Mikko Salo.
\newblock On instability mechanisms for inverse problems.
\newblock {\em Ars Inven. Anal.}, pages Paper No. 7, 93, 2021.

\bibitem[KRZ22]{KRZ2022Biharm}
Manas Kar, Jesse Railo, and Philipp Zimmermann.
\newblock The fractional $p\,$-biharmonic systems: optimal poincaré constants,
  unique continuation and inverse problems, 2022.
\newblock arXiv:2208.09528.

\bibitem[LL22]{LL-fractional-semilinear-problems}
Ru-Yu Lai and Yi-Hsuan Lin.
\newblock Inverse problems for fractional semilinear elliptic equations.
\newblock {\em Nonlinear Anal.}, 216:Paper No. 112699, 2022.

\bibitem[Man01]{MandacheInstability}
Niculae Mandache.
\newblock Exponential instability in an inverse problem for the
  {S}chr\"{o}dinger equation.
\newblock {\em Inverse Problems}, 17(5):1435--1444, 2001.

\bibitem[MS09]{Mazya}
Vladimir~G. Maz'ya and Tatyana~O. Shaposhnikova.
\newblock {\em Theory of {S}obolev multipliers}, volume 337 of {\em Grundlehren
  der mathematischen Wissenschaften [Fundamental Principles of Mathematical
  Sciences]}.
\newblock Springer-Verlag, Berlin, 2009.
\newblock With applications to differential and integral operators.

\bibitem[Nac88]{Na88}
Adrian~I. Nachman.
\newblock Reconstructions from boundary measurements.
\newblock {\em Ann. of Math. (2)}, 128(3):531--576, 1988.

\bibitem[RS96]{RunstSickel+1996}
Thomas Runst and Winfried Sickel.
\newblock {\em Sobolev Spaces of Fractional Order, Nemytskij Operators, and
  Nonlinear Partial Differential Equations}.
\newblock De Gruyter, Berlin, New York, 1996.

\bibitem[RS20]{RS-fractional-calderon-low-regularity-stability}
Angkana R\"{u}land and Mikko Salo.
\newblock The fractional {C}alder\'{o}n problem: low regularity and stability.
\newblock {\em Nonlinear Anal.}, 193:111529, 56, 2020.

\bibitem[RZ23a]{RZ2022counterexamples}
Jesse Railo and Philipp Zimmermann.
\newblock Counterexamples to uniqueness in the inverse fractional conductivity
  problem with partial data.
\newblock {\em Inverse Probl. Imaging}, 17(2):406--418, 2023.

\bibitem[RZ23b]{RZ2022unboundedFracCald}
Jesse Railo and Philipp Zimmermann.
\newblock Fractional {C}alder\'on problems and {P}oincar\'e{} inequalities on
  unbounded domains.
\newblock {\em J. Spectr. Theory}, 13(1):63--131, 2023.

\bibitem[RZ24]{RZ2022LowReg}
Jesse Railo and Philipp Zimmermann.
\newblock Low regularity theory for the inverse fractional conductivity
  problem.
\newblock {\em Nonlinear Anal.}, 239:Paper No. 113418, 27, 2024.

\bibitem[SU87]{SU87}
John Sylvester and Gunther Uhlmann.
\newblock A global uniqueness theorem for an inverse boundary value problem.
\newblock {\em Ann. of Math.}, 125(1):153--169, 1987.

\bibitem[Uhl14]{GU30years}
Gunther Uhlmann.
\newblock 30 years of {C}alder\'{o}n's problem.
\newblock In {\em S\'{e}minaire {L}aurent {S}chwartz---\'{E}quations aux
  d\'{e}riv\'{e}es partielles et applications. {A}nn\'{e}e 2012--2013},
  S\'{e}min. \'{E}qu. D\'{e}riv. Partielles, pages Exp. No. XIII, 25. \'{E}cole
  Polytech., Palaiseau, 2014.

\end{thebibliography}

\bibliographystyle{alpha}

\end{document}